\documentclass[12pt]{article}
\usepackage{amsthm, amsfonts}
\usepackage{url, cells}
\usepackage{amscd,amsmath,amssymb}

\newtheorem{definition}{Definition}
\newtheorem{theorem}{Theorem}
\newtheorem{lemma}{Lemma}
\newtheorem{proposition}{Proposition}
\newtheorem{corollary}{Corollary}

\def\de{\delta}

\def\al{\alpha}
\def\Ga{\Gamma}

\def\La{\Lambda}
\def\la{\lambda}
\def\Om{\Omega}
\def\om{\omega}

\def\kappa{\varkappa}

\def\C{{\mathbb C}}

\def\Z{{\mathbb Z}}

\def\sN{{\mathfrak S}_N}
\def\sl{\mathfrak{sl}_2}
\def\sinf{{\mathfrak S}_{\mathbb N}}
\def\maj{\operatorname{maj}}
\def\des{\operatorname{des}}
\def\Vir{\operatorname{Vir}}
\def\te{{\widetilde e}}

\def\T{{\cal T}}
\def\H{{\cal H}}
\def\X{{\cal X}}
\def\F{{\cal F}}
\def\K{{\cal K}}

\def\beq{\begin{equation}}
\def\eeq{\end{equation}}

\urldef\vershik\url{vershik@pdmi.ras.ru}
\urldef\natalia\url{natalia@pdmi.ras.ru}

\author{N.~V.~Tsilevich\thanks{%
St.~Petersburg Department of Steklov Institute of Mathematics and St.~Petersburg State University.
E-mail: \natalia, \vershik. Supported by the grants
RFBR 13-01-12422-ofi-m and RFBR 14-01-00373-a.}
\and A.~M.~Vershik\footnotemark[1]}
\title{The serpentine representation of the infinite symmetric group and the basic representation of the affine Lie algebra $\widehat{\mathfrak{sl}_2}$}
\date{}

\begin{document}
\maketitle

\begin{abstract}
We introduce and study the so-called serpentine representations of the infinite symmetric group $\sinf$, which turn out to  be closely related to the basic representation of the affine Lie algebra
$\widehat{\mathfrak{sl}_2}$ and representations of the Virasoro algebra.


\end{abstract}

\section{Introduction. Infinite-dimensional Schur--Weyl duality and the serpentine representation of $\sinf$}

The serpentine representation is a remarkable representation of the infinite symmetric group $\sinf$, which has not yet been studied. Its importance is due to the fact that it is very closely related to the basic representation of the affine Lie algebra
$\widehat{\mathfrak{sl}_2}$ and representations of the Virasoro algebra.
This representation belongs to the class of so-called Schur--Weyl representations.
Recall that in \cite{SW} we suggested an infinite-dimensional generalization of the classical Schur--Weyl duality for the symmetric group $\sN$ and the special linear group $SL(2,\mathbb C)$ using a ``dynamical'' approach. Namely, we started from the classical Schur--Weyl duality (for definiteness, assume that $N=2n$)
\beq\label{swfin}
(\C^2)^{\otimes N}
=\sum_{k=0}^{n}M_{2k+1}\otimes H_{\pi_{k}},
\eeq
where $H_{\pi_{k}}$ is the space of the irreducible representation  $\pi_{k}$ of the
symmetric group $\sN$
corresponding to the two-row Young diagram
$(n+k,n-k)$
and $M_{2k+1}$ is the $(2k+1)$-dimensional irreducible
$SL(2,\C)$-module, and considered  {\it
isometric embeddings $(\C^2)^{\otimes
N}\hookrightarrow(\C^2)^{\otimes (N+2)}$
that are equivariant with respect to both the actions of $SL(2,\C)$ and $\sN$}, which we called {\it Schur--Weyl
embeddings.}
Given an infinite chain
\beq\label{even}
(\C^2)^{\otimes0}\stackrel{\al_{0}}{\hookrightarrow}(\C^2)^{\otimes
2}\stackrel{\al_{2}}{\hookrightarrow}(\C^2)^{\otimes
4}\stackrel{\al_{4}}{\hookrightarrow}{\ldots}
\eeq
of Schur--Weyl embeddings, we can consider the corresponding inductive limit.
The class of all representations (called {\it Schur--Weyl representations}) that can be obtained in this way is described in \cite[Theorem~1]{SW}:
 {\it
Let $\Pi^{\{\al_{N}\}}$ be the representation of the infinite symmetric
group $\sinf$ obtained as the inductive limit of the standard
representations of $\sN$ in $(\C^2)^{\otimes N}$ with respect to an infinite
chain of Schur--Weyl embeddings~\eqref{even}. Then it
decomposes into a countable direct
sum of primary representations
\beq\label{X}
\Pi^{\{\al_{N}\}}=\sum_{k=0}^\infty M_{2k+1}\otimes\Pi_{k}^{\{\al_{N}\}} ,
\eeq
where $\Pi_{k}^{\{\al_{N}\}}$ is  the inductive limit of the irreducible
representations of ${\mathfrak S}_{2k},\break
{\mathfrak S}_{2k+2},{\ldots} $ corresponding to
the Young diagrams
$(2k),(2k+1,1),(2k+2,2),\ldots$.}

As an important example of such a representation, in \cite{SW} we considered the unique infinite Schur--Weyl scheme that satisfies a natural additional condition, namely, preserves the tensor structure of $(\C^2)^{\otimes N}$.
{\it The main goal of this paper is to study another example of Schur--Weyl duality,
namely, the {\bf unique} Schur--Weyl scheme that satisfies the following additional
condition: it preserves the so-called stable major index of a Young tableau.} We show that this particular representation of the infinite symmetric group, which we call the serpentine representation, can be naturally equipped with the structure of the basic representation of the affine Lie algebra $\widehat{\mathfrak{sl}_2}$, with the irreducible $\sN$-modules corresponding to the irreducible Virasoro modules. This reveals new interrelations between the representation theory of the infinite symmetric group and that of the
affine Lie and Virasoro algebras. The precise form of the underlying natural grading-preserving isomorphism of $\sl$-modules is still unknown in the general case, and perhaps it is not a simple task to find it, but we present several properties of this isomorphism which are corollaries of the main theorem.

Our approach uses the result of  \cite{FF} that the level~1 irreducible
highest weight representations of $\widehat\sl$ can be realized as certain inductive limits of tensor powers
$(\C^{2})^{\otimes N}$ of the two-dimensional irreducible representation of $\sl$. The construction of \cite{FF}  is based on the notion of the fusion product of representations, whose
main ingredient is, in turn, a special grading in the space $(\C^{2})^{\otimes N}$. A key observation underlying the results of this paper, which relies on the computation presented in \cite{Kedem} of the $q$-characters of the multiplicity spaces of irreducible $\sl$-modules with respect to this grading, is that the fusion product under consideration can be realized in
an $\sN$-module so that this special grading essentially coincides with a well-known
combinatorial characteristic of Young tableaux called the major index (see Proposition~\ref{prop:fin}). Thus
our results provide, in particular, a kind of combinatorial description of the fusion product and show that
the combinatorial notion of the major index of a Young tableau has a new  representation-theoretic meaning. For instance, Corollary~\ref{cor:vir} in Sec.~\ref{sec:iso}  shows that {\it the so-called stable major indices of infinite Young tableaux are the eigenvalues of the Virasoro $L_{0}$ operator, the Gelfand--Tsetlin basis of the Schur--Weyl module being its eigenbasis.}

The paper is organized as follows. In Sec.~\ref{sec:main} we introduce our main object, the so-called serpentine representation of the infinite symmetric group, as well as the notion of the stable major index of an infinite Young tableau, and formulate our main Theorem~\ref{th2},
which states that there is a grading-preserving isomorphism of $\sl$-modules between the basic $\widehat\sl$-module
$L_{0,1}$ and the space $H_{\Pi}$ of the serpentine representation. The theorem is proved in Sec.~\ref{sec:proof}. In Sec.~\ref{sec:iso} we study the above isomorphism in more detail, describing some of its properties and giving examples.

\smallskip
{\it For definiteness, in what follows we consider only the even case
$N=2n$. The odd case can be treated in exactly the same way; instead
of the basic representation $L_{0,1}$, it leads to the other level~$1$ highest weight representation
$L_{1,1}$ of $\widehat\sl$.}
\bigskip

\noindent {\bf Acknowledgments.} The authors are grateful to I.~Frenkel, B.~Feigin, and E.~Feigin for useful discussions.
We also thank the anonymous referees for valuable comments and suggestions which helped us to substantially improve the paper.

\section{The main theorem}\label{sec:main}

\subsection{The serpentine representation, stable major index, and the statement of the theorem}
Let $T_N$ be the set of all standard Young tableaux with $N$ cells
and at most two rows.
Consider the following natural embedding $i_N:T_N\to
T_{N+2}$: given a standard Young tableau $\tau$ with $N$ cells, its
image $i_N(\tau)$ is the standard Young tableau with $N+2$ cells
obtained from $\tau$  by adding the element $N+1$ to the first row and
the element $N+2$ to the second row.
As shown in \cite{SW}, it determines a Schur--Weyl
embedding  $(\C^2)^{\otimes
N}\hookrightarrow(\C^2)^{\otimes (N+2)}$, which, by abuse of notation, we denote by the same symbol $i_{N}$.

\begin{definition}
The Schur--Weyl representation $\Pi:=\Pi^{(i_{N})}$
of the infinite symmetric group $\sinf$ in the space $H_{\Pi}=\lim((\C^2)^{\otimes
N}, i_{N})$ will be called the serpentine representation.
\end{definition}

According to the theorem on Schur--Weyl representations (see the introduction), we have
\beq\label{serp}\Pi=\sum_{k=0}^\infty M_{2k+1}\otimes\Pi_{k},
\eeq
where the irreducible component $\Pi_k$, which will be called the {\it $k$-serpentine} representation, is
the representation of  $\sinf$
associated with the infinite tableau
\begin{center}
$\tau_{k}=\lower1\cellsize\vbox{\footnotesize
\cput(1,1){1}
\cput(1,2){2}
\cput(1,3){$\ldots$}
\cput(1,4){$2k$}
\cput(1,5){$2k+1$}
\cput(1,6){$2k+3$}
\cput(1,7){$\ldots$}
\cput(2,1){$2k+2$}
\cput(2,2){$2k+4$}
\cput(2,3){$\ldots$}
 \cells{
 _ _ _ _ _ _ _
|_|_|_|_|_|_|_|
|_|_|_|}}$
\end{center}

\noindent (in particular, $\tau_{0}$ is the tableau with $1,3,5,\ldots$ in the first row and $2,4,6,\ldots$ in the second row),
which can be realized in the space $H_{\Pi_k}$ spanned by the set $\T_{k}$ of infinite two-row Young tableaux
tail-equivalent to $\tau_k$. It has a discrete spectrum with respect to the Gelfand--Tsetlin algebra.

In what follows,
the tableaux $\tau_k$, $k=0,1,\ldots$, will be called {\it principal}, and a tableau tail-equivalent to $\tau_{k}$ for some $k$ will be called a {\it serpentine tableau}; denote by $\T=\cup\T_{k}$ the set of all serpentine tableaux.

Now consider the well-known statistic on
Young tableaux called the {\it major index}. It is defined as
follows (see \cite[Sec.~7.19]{St}):
$$
\maj(\tau)=\sum_{i\in\des(\tau)} i,
$$
where, for $\tau\in T_{N}$,
$$
\des(\tau)=\{i\le N-1: \mbox{ the element $i+1$  in $\tau$ lies lower than $i$}\}
$$
is the descent set of $\tau$.

Obviously,
\beq\label{embmaj}
\maj(i_N(\tau))=\maj(\tau)+(N+1).
\eeq
This suggests the following important step. Given $N=2n$ and $\tau\in T_N$,
denote $r_N(\tau)=n^2-\maj(\tau)$. Then
$r_{N+2}(i_N(\tau))=r_N(\tau)$, so that we have a well-defined index on all serpentine tableaux $\tau\in\T$:
\beq\label{r}
r(\tau)=\lim_{n\to\infty}r_{2n}([\tau]_{2n})=\lim_{n\to\infty}(n^2-\maj([\tau]_{2n})),
\eeq
where $[\tau]_l$ is the tableau with $l$ cells obtained from $\tau$ by removing all the cells with entries $k>l$. Obviously, for the principal tableaux we have $r(\tau_k)=k^2$.

\begin{definition}We call $r(\tau)$ the {\it  stable major index} of an infinite tableau $\tau\in\T$.
\end{definition}

The stable major index determines a grading on all the spaces $H_{\Pi_k}$ and hence on the whole space $H_{\Pi}$: for $w=u\otimes v\in M_{2k+1}\otimes H_{\Pi_{k}}$ we just set $\deg_{r}(w)=r(v)$.

Now consider the affine
Lie algebra $\widehat\sl=\sl\otimes\C[t,t^{-1}]\oplus\C c\oplus\C d$, its basic module $L_{0,1}$ with the homogeneous grading $\deg_{H}$, and the natural embedding $\sl\subset\widehat\sl$ given by $\sl\supset x\mapsto x\otimes1\in\widehat\sl$.
Our main theorem is the following.

\begin{theorem}\label{th2}
There is
a grading-preserving unitary isomorphism of $\sl$-modules between $(L_{0,1},\deg_{H})$
and $(H_{\Pi},\deg_{r})$. The serpentine representation is the unique Schur--Weyl representation satisfying this condition.
\end{theorem}

\smallskip\noindent{\bf Remarks. 1.}
As mentioned in the introduction, we consider only the even case
just for simplicity of notation. Considering instead of \eqref{even} the chain
$(\C^2)^{\otimes1}\hookrightarrow(\C^2)^{\otimes
3}\hookrightarrow(\C^2)^{\otimes
5}\hookrightarrow{\ldots}$ and reproducing exactly
the same arguments, we will obtain a grading-preserving isomorphism of the corresponding Schur--Weyl module with the other level~$1$ highest weight module
$L_{1,1}$ of $\widehat\sl$.

\medskip\noindent{\bf 2.} The conditions from the statement of Theorem~\ref{th2} do not uniquely determine the isomorphism, since there is a nontrivial group of transformations in $H_{\Pi}$ that commute with $\sl$ and preserve the grading. For more details, see the remark after Corollary~\ref{cor:vir} in Sec.~\ref{sec:iso}. To find an explicit form of this isomorphism is an intriguing problem.

\section{Proof of the main theorem}\label{sec:proof}

 \smallskip\noindent{\bf 1. Fusion product.} Our proof relies on the result of B.~Feigin and E.~Feigin \cite{FF} on a finite-dimensional approximation of
the basic representation of $\widehat\sl$, which, in turn, uses the notion of
 the fusion product of representations
introduced in \cite{FL}. Since the corresponding construction is of importance for us, we describe it in some detail.

Given a representation $\rho$ of $\sl$ and
$z\in\C$, let $\rho(z)$ be the evaluation representation of the polynomial
current algebra $\sl\otimes\C[t]$, defined as\break $(x\otimes
t^i)v=z^i\cdot xv$. Now, given a collection
$\rho_1,{\ldots} ,\rho_N$ of irreducible representations of $\sl$
with lowest weight vectors $v_1,{\ldots} ,v_N$, and
a collection $z_1,{\ldots} ,z_N$ of pairwise distinct complex numbers,
we consider the tensor product of the corresponding evaluation
representations: $\rho_1(z_1)\otimes{\ldots} \otimes\rho_N(z_N)$.
The crucial step is introducing a special grading in the space $V_N$ of this representation. Set
$
V_N^{(m)}=U^{(m)}(e\otimes\C[t])(v_1\otimes{\ldots} \otimes
v_N)\subset V_N
$,
where $e$ 
is the raising
operator in $\sl$ and
$U^{(m)}$ is spanned by homogeneous elements of degree $m$ in $t$. In
other words, $U^{(m)}$  is spanned by the monomials of the form
$e_{i_1}{\ldots} e_{i_k}$  with $i_1+{\ldots} +i_k=m$,
where $e_j=e\otimes t^j$. Then we consider the
corresponding filtration on $V_N$:
$
V_N^{(\le m)}=\sum_{k\le m} V_N^{(k)}
$
The fusion product of $\rho_1,{\ldots} ,\rho_N$ is the graded
representation with respect to the above filtration, which is realized in the space
\beq\label{gr}
V_N^*=\operatorname{gr}V_N=V_N^{(\le0)}\oplus V_N^{(\le1)}/V_N^{(\le0)}\oplus
V_N^{(\le2)}/V_N^{(\le1)}\oplus{\ldots} .
\eeq
The space $V_N^*[k]=V_N^{(\le k)}/V_N^{(\le k-1)}$ is the subspace of
elements of degree $k$, and elements of the form $x\otimes
t^l\in\sl\otimes \mathbb C[t]$ send $V_N^*[k]$ to $V_N^*[k+l]$.
The degree of an element with respect to this grading will be denoted
by $\widetilde\deg$.
It is proved in \cite{FL} that $V_N^*$ is an
$\sl\otimes(\C[t]/t^N)$-module that does not depend on $z_1,{\ldots}
,z_N$ provided that they are pairwise distinct. Moreover, $V_N^*$ is
isomorphic to $\rho_1\otimes{\ldots} \otimes\rho_N$ as an $\sl$-module.

We apply this construction to the case where
$\rho_1={\ldots} =\rho_N$ is the two-dimensional
irreducible representation of $\sl$ with the lowest weight
vector $v_0$.
In this case,
$V_N^*\simeq(\C^2)^{\otimes N}$ as an
$\sl$-module.
We equip $V_N^*$ with the inner product such that the corresponding representation
of $\sl$ is unitary. It is proved in \cite{FF} that an inductive limit of $V_N^*$
is isomorphic to the basic representation $L_{0,1}$ of
 $\widehat\sl$, so that we first establish a grading-preserving isomorphism of the finite-dimensional $\sl$-modules $V_N^*$ and $\sum_{k=0}^{n}M_{2k+1}\otimes H_{\pi_{k}}$ and then show that it can be extended to the inductive limits of the corresponding spaces.

\smallskip\noindent{\bf 2. The $q$-character, major index, and the finite-dimensional result.} Consider the decomposition of $V_N^*$ into irreducible $\sl$-modules:
$$
V_N^*=\bigoplus_{k=0}^nM_{2k+1}\otimes{\cal M}_k.
$$
By the
classical Schur--Weyl duality~\eqref{swfin}, we know that the
multiplicity space ${\cal M}_k$ coincides with the space $H_{\pi_{k}}$ of the irreducible representation of $\sN$ with the Young diagram $(n+k,n-k)$.
On the other hand, it
inherits the grading from $V_N^*$:
\beq\label{mult}
{\cal M}_k=\bigoplus_{i\ge0}{\cal M}_k[i],
\eeq
where ${\cal M}_k[i]={\cal M}_k\cap V^*[i]$.
Consider the corresponding $q$-character
$$
\operatorname{ch}_q{\cal M}_k=\sum_{i\ge0}q^i\dim{\cal M}_k[i].
$$
It was proved by Kedem \cite{Kedem} that
\beq\label{kedem}
\operatorname{ch}_q{\cal M}_k=q^{\frac{N(N-1)}2}\cdot K_{(n+k,n-k),1^N}(1/q),
\eeq
where $K_{\la,\mu}$ is the Kostka--Foulkes polynomial (see
\cite[Sec.~III.6]{Mac}).

Now we use the well-known combinatorial description of the Kostka--Foulkes polynomial due to Lascoux and Sch\"utzenberger \cite{LS}.
For a two-row partition $\la$, their formula reduces to
\beq\label{LS}
K_{\la,1^N}(q)=\sum_{\tau\in[\la]}q^{c(\tau)},
\eeq
where  $[\la]$ is the set of standard Young tableaux of shape $\la$
and $c(\tau)$ is the charge of a tableau $\tau\in T_N$, defined
as the sum of $i\le N-1$ such that in $\tau$ the element $i+1$ lies to
the right of $i$ (see \cite{Mac}).
But, obviously, for $\tau\in T_N$ we have
$\maj(\tau)=\frac{N(N-1)}2-c(\tau)$.
Then it follows from
\eqref{kedem} and \eqref{LS} that
\beq\label{dim}
\dim{\cal M}_k[i]=\#\{\tau\in[(n+k,n-k)]:\maj(\tau)=i\}.
\eeq

The major index defines a grading in the space $H_{\pi_{k}}$ (spanned by the standard Young tableaux of shape $(n+k,n-k)$), and hence in the whole space
$\X_{N}
=\sum_{k=0}^{n}M_{2k+1}\otimes H_{\pi_{k}}$, which we
equip with the standard inner product. We obtain the following finite-dimensional analog of Theorem~\ref{th2}.

\begin{proposition}\label{prop:fin}
There is a grading-preserving unitary isomorphism of $\sl$-modules between
$(V_N^*,\widetilde\deg)$ and $(\X_{N},\maj)$ such that
the multiplicity space ${\cal M}_k$ is spanned by the standard Young tableaux $\tau$ of shape
$(n+k,n-k)$ (and hence  ${\cal M}_k[i]$ is spanned by $\tau$ with $\maj(\tau)=i$).
\end{proposition}
\begin{proof}
Follows from the fact that the fusion product
$V_N^*$ is isomorphic to $(\C^2)^{\otimes N}$ as an $\sl$-module and
equation~\eqref{dim}.
\end{proof}

\smallskip\noindent{\bf Remarks. 1.}
Observe that the isomorphism from Proposition~\ref{prop:fin} is not
unique.

\noindent{\bf 2.} The isomorphism from Proposition~\ref{prop:fin} determines
an action of the symmetric group $\sN$ on the space $V_N^*$. It does
not coincide with the original action of $\sN$ on $\C^{\otimes N}$.
\medskip

\smallskip\noindent{\bf 3. Embeddings and the limit.}
It is proved in \cite{FF} that there is an embedding
$
j_N:V_N^*\to V^*_{N+2}
$
equivariant with respect to the action of $\sl\otimes(\C[t^{-1}]/t^{-n})$, and
the corresponding inductive limit
$
{\cal V}=\lim(V_N^{*},j_N)
$
is isomorphic to the basic representation $L_{0,1}$ of
the affine Lie algebra $\widehat\sl$.
This embedding satisfies
\beq\label{embdeg}
\widetilde\deg(j_Nx)=\widetilde\deg(x)-(N+1).
\eeq

Since we are now considering
$\sl\otimes\C[t^{-1}]$ instead of $\sl\otimes\C[t]$, we should
slightly modify the previous constructions to take the minus sign into
account. Namely,
instead of~\eqref{mult} we now have
${\cal M}_{k}=\oplus_{i\ge0}{\cal M}_{k}[-i]$, and the isomorphism of
Proposition~\ref{prop:fin} identifies ${\cal M}_{k}[-i]$ with the space spanned
by the tableaux $\tau$ of
shape $(n+k,n-k)$ such that $\maj(\tau)=i$. Denote this isomoprhism
between $V^*_N$ and $\X_N$ by $\rho_N$. Observe that the only
conditions we impose on $\rho_N$ are as follows: (a) $\rho_N$ is a
unitary isomorphism
of $\sl$-modules and (b) $\rho_N\circ\widetilde\deg=-\maj$.

Now, since $L_{0,1}\simeq\lim(V_N^{*},j_N)$, $H_{\Pi}=\lim(\X_{N}, i_{N})$, and Proposition~\ref{prop:fin} holds,
in order to prove Theorem~\ref{th2} it suffices to show that we can choose
a sequence of isomorphisms $\rho_N$ such that the diagram
$$
\begin{CD}
V_N^*@>\rho_N>>\X_N\\
@VVj_NV @ VVi_NV\\
V_{N+2}^*@>\rho_{N+2}>>\X_{N+2}
\end{CD}
$$
is commutative for all $N$. We use induction on $N$. The base being
obvious, assume that we have already constructed $\rho_N$, and let us
construct $\rho_{N+2}$.

We have $V_{N+2}^*=j_N(V_N^*)\oplus (j_N(V_N^*))^\perp $. On the first subspace, we
set $\rho_{N+2}(x):=i_N(\rho_N(j_{N+2}^{-1}(x)))$. On the second
one, we define it in an arbitrary way to satisfy the desired
conditions~(a) and~(b). The fact that this definition is correct and
provides us with a desired isomorphism between $V_{N+2}^*$ and
$\X_{N+2}$ follows
from~\eqref{embdeg} and~\eqref{embmaj}. The theorem is proved.

\section{The key isomorphism in more detail}\label{sec:iso}

Our aim in this section is to study the isomorphism from Theorem~\ref{th2} in more detail. For this, we first give necessary background on the Fock space realization of the basic $\widehat\sl$-module.

\subsection{The Fock space realization of the basic $\widehat\sl$-module}\label{sec:fock}

Let $\F$ be the fermionic Fock space constructed as the infinite wedge
space over the linear space with basis
$\{u_k\}_{k\in\Z}\cup\{v_k\}_{k\in\Z}$.
That is, $\F$ is spanned by the semi-infinite forms
$u_{i_1}\wedge{\ldots}\wedge u_{i_k}\wedge v_{j_1}\wedge{\ldots}\wedge v_{j_l}\wedge
u_N\wedge v_N\wedge u_{N-1}\wedge v_{N-1}\wedge{\ldots}$,
$N\in\Z$, $i_1>{\ldots} >i_k>N$, $j_1>{\ldots} >j_l>N$,
and is equipped with the inner product in which such monomials are
orthonormal.
Let $\phi_k$ be the exterior multiplication by $u_k$
and $\psi_k$ be the exterior multiplication by $v_k$,
and denote by $\phi_k^*$, $\psi_k^*$ the corresponding adjoint
operators. Then this family of operators satisfies the canonical
anticommutation relations (CAR).
We consider the generating functions
$\phi(z)=\sum_{i\in\Z}\phi_iz^{-(i+1)}$,
$\phi^*(z)=\sum_{i\in\Z}\phi_i^*z^i$, and the same for $\psi$ and $\psi^{*}$.

Let $a^\phi_n$ and $a^\psi_n$
be the systems of bosons constructed from the fermions $\{\phi_k\}$ and
$\{\psi_k\}$, respectively: $a^\phi_n=\sum_{k\in\Z}\phi_{k}\phi^*_{k+n}$ for $n\ne0$ and
$a^\phi_0=\sum_{n=1}^\infty\phi_{n}\phi_{n}^*-\sum_{n=0}^\infty\phi_{-n}^*\phi_{-n}$,
and similarly for $a^\psi$. They satisfy the canonical commutation
relations (CCR),
i.e., form a representation of the Heisenberg algebra $\mathfrak A$.
Denote
$a^\phi(z)=\sum_{n\in\Z}a^\phi_nz^{-(n+1)}$, and similarly for $a^{\psi}$.

Let $V$ be the operator in $\F$ that shifts the indices by 1:
$$
V(w_{i_1}\wedge w_{i_2}\wedge{\ldots} )=V_0(w_{i_1})\wedge
V_0(w_{i_2})\wedge{\ldots},\quad V_0(u_i)=u_{i+1},\quad
V_0(v_i)=v_{i-1}.
$$

The vacuum vector in $\F$ is $\Om=u_{-1}\wedge v_{-1}\wedge
u_{-2}\wedge v_{-2}\wedge{\ldots} $. We also consider the family of
vectors
$$
\Om_0=\Om,\quad \Om_{2n}=V^{-n}\Om_0,\quad n\in\Z.
$$

In the space $\F$ we have a canonical representation of the affine
Lie algebra $\widehat\sl=\sl\otimes\C[t,t^{-1}]\oplus\C c\oplus\C d$, which is given by the following formulas.
Given $x\in\sl$, denote $X(z)=\sum_{n\in\Z}x_nz^{-(n+1)}$. Then
\begin{eqnarray*}
E(z)=\psi(z)\phi^*(z),\qquad
F(z)=\phi(z)\psi^*(z),\\
h_n=a^\psi_{-n}-a^\phi_{-n},\qquad
d=\frac{h_0^2}2+\sum_{n=1}^\infty h_{-n}h_n,\qquad c=1.
\end{eqnarray*}

We have
$
\F=\H_0\otimes\K_0+\H_1\otimes\K_1$,
where $\H_0\simeq L_{0,1}$ and $\H_1\simeq L_{1,1}$ are the
irreducible  level~1 highest weight
$\widehat\sl$-modules and  $\K_0$ and $\K_1$ are the multiplicity spaces.
Observe also that
$e_{-(N+1)}\Om_{-N}=\Om_{-(N+2)}$.

Note that the operators $a_n=\frac1{\sqrt{2}}h_n$
satisfy the CCR, i.e., form a system of free bosons, or
generate the Heisenberg algebra ${\mathfrak A}_h$.
The vectors $\{\Om_{2n}\}_{n\in\Z}$ introduced above are exactly singular vectors for
this Heisenberg algebra:
$h_k\Om_m=0$ for $k>0$, $h_0\Om_m=m\Om_m$. The representation of
${\mathfrak A}_h$ in $\H_0$ breaks into a direct sum of
irreducible representations:
\beq
\H_0=\bigoplus_{k\in\Z}\H_0[2k],
\eeq
where $\H_0[2k]$ is the charge~$2k$ subspace, i.e., the eigenspace of
$h_0$ with eigenvalue $2k$:
$$\H_0[2k]=\{v\in\H_0:h_0v=2kv\}=\C[h_0,h_1,{\ldots} ]\Om_{2k}.
$$

Now, given a representation of the affine Lie algebra $\widehat\sl$, we can
use the Sugawara construction to obtain the corresponding
representation of the Virasoro algebra $\Vir$. It can also be
described in the following way.  As noted above, the operators $a_n=\frac1{\sqrt{2}}h_n$
form a system of free bosons. Given
such a system, a representation of $\Vir$ can be constructed as
follows (\cite{Vir}; see also \cite[Ex.~9.17]{Kac}):
\beq\label{sug}
L_n=\frac12\sum_{j\in\Z}a_{-j}a_{j+n},\quad n\ne0;\qquad
L_0=\sum_{j=1}^\infty a_{-j}a_j.
\eeq
Thus we obtain a representation of $\Vir$ in $\F$ and, in particular,
in $\H_0$. In this representation, the algebras generated by the
operators of $\Vir$ and
$\sl\subset\widehat\sl$ are mutual commutants, and we have the decomposition
\beq\label{virdecomp}
\H_0=\bigoplus_{k=0}^\infty M_{2k+1}\otimes L(1,k^2),
\eeq
where $M_{2k+1}$ is the $(2k+1)$-dimensional irreducible $\sl$-module
and $L(1,k^2)$ is the irreducible Virasoro module with
central charge~1 and conformal dimension $k^2$.

The charge $k$ subspace $\H_0[k]$ contains a series of singular
vectors $\xi_{k,m}$ of $\Vir$  with energy $(k+m)^2$:
$$
L_n\xi_{k,m}=0 \mbox{ for }n=1,2,{\ldots} ,\qquad L_0\xi_{k,m}=(k+m)^2.
$$

Let us use the so-called
homogeneous vertex operator construction of the basic
representation of $\widehat\sl$ (\cite{FK}, see also \cite[Sec.~14.8]{Kac}). In this realization,
\beq\label{EF}
E(z)=\Ga_-(z)\Ga_+(z)z^{-h_0}V^{-1},\qquad
F(z)=\Ga_+(z)\Ga_-(z)z^{h_0}V,
\eeq
where
$$
\Ga_\pm(z)=\exp\left(\mp\sum_{j=1}^\infty \frac{z^{\pm j}}jh_{\pm j}\right)
$$
and the operators $\Ga_\pm(z)$ satisfy the commutation relation
\beq\label{Ga}
\Ga_+(z)\Ga_-(w)=\Ga_-(w)\Ga_+(z)\left(1-\frac zw\right)^2.
\eeq

Using the
boson--fermion correspondence (see \cite[Ch.~14]{Kac}), we can
identify $\H_0$ with the space $\La\otimes\C[q,q^{-1}]$, where $\La$ is
the algebra
of symmetric functions (see \cite{Mac}). In particular, consider
the charge~0 subspace $\H[0]=\H_0[0]$, which is
identified with $\La$. We can use the following representation of the
Heisenberg algebra generated by $\{h_n\}_{n\in\Z}$:
\beq\label{power}
h_n\leftrightarrow 2n\frac\partial{\partial p_n},\qquad
h_{-n}=p_n,\qquad n>0,
\eeq
where $p_j$ are Newton's power sums.
Note that the representation~\eqref{power} of ${\mathfrak A}_{h}$,
and hence the corresponding representation~\eqref{sug} of $\Vir$, are
not unitary with respect to the standard inner product in $\La$. To
make it unitary, we should consider the inner product in $\La$ defined
by
\beq\label{inner}
\langle p_\la,p_\mu\rangle=\de_{\la\mu}\cdot z_\la\cdot 2^{l(\la)},
\eeq
where $p_\la$ are the power sum symmetric functions, $z_\la=\prod_i
i^{m_i}m_i!$ for a Young diagram $\la$ with $m_i$ parts of length $i$, and
$l(\la)$ is the number of nonzero rows in $\la$.

Denote the singular vectors of $\Vir$ in $\H[0]$ by
$\xi_m:=\xi_{0,m}$. According to a result by Segal \cite{Segal}, in
the symmetric function realization~\eqref{power},
\beq\label{segal}
\xi_{n}\leftrightarrow c\cdot s_{(n^n)},
\eeq
where $s_{(n^n)}$ is the Schur function indexed by the $n\times n$
square Young diagram and $c$ is a numerical coefficient.

\subsection{Further analysis of the key isomorphism}\label{subsec:iso}

Comparing the structure~\eqref{serp} of the serpentine representation with \eqref{virdecomp}, we obtain the
following result.

\begin{corollary}
The space $H_{\Pi_k}$ of the $k$-serpentine representation of the infinite
symmetric group  has a natural
structure of the Virasoro module $L(1,k^2)$.
\end{corollary}

Our aim is to study this Virasoro representation in $\Pi_k$ (or, which
is equivalent, the corresponding representation of the infinite
symmetric group in the Fock space). In particular, from the known
theory of the basic module $L_{0,1}$, we immediately obtain the following
result.

\begin{corollary}\label{cor:vir}
In the above realization of the Virasoro module $L(1,k^2)$,
the Gelfand--Tsetlin basis of $H_{\Pi_k}$ (which consists of the infinite
two-row Young tableaux tail-equivalent to $\tau_k$) is the eigenbasis of
$L_0$, and the eigenvalues are given by the stable major index $r$:
$$
L_0\tau=r(\tau)\tau.
$$
\end{corollary}

Now we see that the isomorphism in Theorem~\ref{th2} is in fact defined up to the commutant of $L_{0}$ in each $\Pi_{k}$.

Let $\om_{-2k}$ be the lowest vector in $M_{2k+1}$. Then
a natural basis of $\cal V$ is
$
\{e_0^m\om_{-2k}\otimes\tau: m=0,1,{\ldots} ,2k,\;\tau\in\T_k\}
$.
Denoting ${\cal V}_k=M_{2k+1}\otimes H_{\Pi_k}$ and ${\cal
V}_k[0]=\{v\in{\cal V}_k:h_0v=0\}$, we have ${\cal
V}_k[0]=e_0^k\om_{-2k}\otimes H_{\Pi_k}$, so that
we may identify
${\cal V}_k[0]$ with $H_{\Pi_k}$ via the correspondence
$$
c(t)\cdot e_0^k\om_{-2k}\otimes t\leftrightarrow t,
$$
where $c(t)$ is a normalizing constant. On the other hand, it is shown in \cite{FF} that
$V_{2n}^*\simeq\C[e_0,{\ldots} ,e_{-(2n-1)}]\Om_{-2n}\subset\F$
as an $\sl\otimes(\C[t^{-1}]/t^{-2n})$-module,
and the limit space $\cal V$ coincides
with $\H_0$. Under this  correspondence,
the charge~0 subspace $\H[0]$ is identified with ${\cal
V}[0]=\{v\in{\cal V}:h_0v=0\}$.
Thus we have
\beq\label{corr0}
\H[0]
\simeq H_{\Pi}[0]=\bigoplus_{k=0}^\infty H_{\Pi_k},
\eeq
where $H_{\Pi}[0]$ is the space spanned by all serpentine tableaux, which is the space of the countable sum of the $k$-serpentine representations $\Pi_{k}$ of $\sinf$ without multiplicities, and the following corollary holds.

\begin{corollary}
The space  $H_{\Pi}[0]$ has a
structure of an irreducible representation of the Heisenberg algebra
$\mathfrak A$.
\end{corollary}

Now, using results of \cite{FF}, one can easily prove the following lemma.

\begin{lemma}\label{l:ebasis}
A basis in $F_{2n}=\C[e_0,{\ldots},e_{-(2n-1)}]\Om_{-2n}$ is
$$
\{e_0^{i_0}e_{-1}^{i_1}{\ldots} e_{-(2n-1)}^{i_{2n-1}}:0\le k\le
2n-(i_0+{\ldots} +i_{2n-1})\}\Om_{-2n}.
$$
In particular, a basis of $F_{2n}[0]=F_{2n}\cap\H[0]$ is
\beq\label{ebasis0}
\{\prod e_0^{i_0}e_{-1}^{i_1}{\ldots}e_{-n}^{i_n}:i_0+i_1+{\ldots}
+i_{n}=n\}\Om_{-2n}.
\eeq
\end{lemma}

On the other hand, as mentioned
above, $\H[0]$ can be identified with the algebra of symmetric
functions $\La$ via~\eqref{power}.
Denote by $\Phi$ the obtained isomorphism between $H_{\Pi}[0]$ and $\La$, which thus associates with every serpentine tableau $\tau\in\T$ a symmetric function $\Phi(\tau)\in\La$ such
that $r(\tau)=\deg\Phi(\tau)$.

\begin{proposition}\label{prop:segal}
Under the isomorphism $\Phi$, the principal
tableaux correspond to the Schur functions with square Young
diagrams:
$$\Phi(\tau_k)={\rm const}\cdot s_{(k^k)}.
$$
\end{proposition}
\begin{proof}
Follows from Segal's \cite{Segal} result~\eqref{segal}, since
it is not difficult to see that the singular vector of $\Vir$ in ${\cal
V}_k[0]$ is just $e_0^k\om_{-2k}\otimes\tau_k$.
\end{proof}

Denote
by $T^{(N)}$ the (finite) set of infinite two-row tableaux that coincide with some
$\tau_n$, $n=0,1,{\ldots} $, from the $N$th level.

\begin{proposition}\label{prop:square}
Let $H_{\Pi}^{(N)}$ be
the subspace in $H_{\Pi}[0]$ spanned by all $\tau\in T^{(N)}$. Then
$$
\Phi(H_{\Pi}^{(2k)})=\La_{k\times k},
$$
where $\La_{k\times k}$ is the subspace in $\La$ spanned by the
Schur functions indexed by Young diagrams lying in the $k\times k$
square.
\end{proposition}
\begin{proof}
From all the above identificatons,
$H_{\Pi}^{(2k)}\leftrightarrow F_{2k}[0]$. Now the claim follows from the result
proved in \cite{FF1}  that in the symmetric functions realization $F_{2k}[0]$
corresponds to $\La_{k\times k}$.
\end{proof}

In the next theorem we refine this result, giving an explicit formula for the Schur  basis in
$\La_{k\times k}$ in terms of the basis~\eqref{ebasis0} in $H_{\Pi} ^{(2k)}\simeq
F_{2k}[0]$. In fact, we would like to have an explicit formula for $\Phi$ or $\Phi^{-1}$, expressing, say, a Schur function in terms of serpentine tableaux. At the moment we cannot provide such a general formula, but the theorem below is a step toward solving this problem, reducing it to describing the action of the operators $e_{-m}$ in the space of serpentine tableaux. Besides, Propositions~\ref{prop:segal} and \ref{prop:square} can easily be derived from formula~\eqref{gensegal}, the former by taking $\nu=(k^{k})$ and the latter by counting the dimensions.

\begin{theorem}\label{th3}
In the symmetric functions realization, the correspondence between the Schur function basis in
$\La_{k\times k}$ and the basis~\eqref{ebasis0} in $H_{\Pi}^{(2k)}\simeq
F_{2k}[0]$ is given by
\beq\label{gensegal}
s_\nu=\sum_{\mu=(0^{r_0}1^{r_1}2^{r_2}{\ldots} )\subset(k^k)}\frac{K_{\nu\mu}}
{\prod_{j=0}^k r_j!}e_{-(k-\mu_1)}{\ldots}
e_{-(k-\mu_k)}\Om_{-2k},\qquad\nu\subset(k^{k}),
\eeq
where $K_{\la\mu}$ are Kostka numbers.
\end{theorem}

\begin{proof}
We generalize Wasserman's \cite{Was} proof of Segal's
result~\eqref{segal} (a similar computation is also given in an earlier paper~\cite{BFJ}).

Let $0\le i_1,{\ldots} ,i_k\le k$. Then, obviously,
$$
e_{-i_1}{\ldots} e_{-i_k}\Om_{-2k}=\left[\prod_{j=1}^kz_j^{i_j-1}\right]
E(z_k){\ldots} E(z_1)\Om_{-2k},
$$
where by $[\mbox{monomial}]F(z_1,{\ldots} ,z_m)$ we denote the coefficient of this
monomial in $F(z_1,{\ldots} ,z_m)$ (in particular, $[\mbox{1}]F(z_1,{\ldots} ,z_m)$ is the constant term of $F$). Now, using the
representation~\eqref{EF}, the commutation relation~\eqref{Ga}, and
the obvious facts that $V^{-k}\Om_{-2k}=\Om_0$ and
$\Ga_+(z)\Om_0=\Om_0$, we obtain
$$
E(z_k){\ldots}E(z_1)\Om_{-2k}=\prod_{j=1}^kz_j^{2(k-j)}
\prod_{1\le j<i\le k}\left(1-\frac{z_i}{z_j}\right)^2\Ga_-(z_k){\ldots}
\Ga_-(z_1)\Om_0.
$$
Observe that, in view of~\eqref{power} and the well-known fact from
the theory of symmetric functions, $\Ga_-(z)$ is
exactly the generating function of the complete symmetric functions. Hence,
expanding the product $\Ga_-(z_k){\ldots}\Ga_-(z_1)\Om_0$ by the Cauchy identity
(\cite[I.4.3]{Mac}) and making simple transformations, we obtain
$$
E(z_k){\ldots}E(z_1)\Om_{-2k}=
(-1)^{k(k-1)/2}
\prod_{j=1}^kz_j^{k-1}a_\de(z)a_\de(z^{-1})
\sum_{\la:\,l(\la)\le k} s_\la(z^{-1})s_\la,
$$
where
$$
a_\de(z)=\prod_{1\le i<j\le k}(z_i-z_j)=\det[z_i^{k-j}]_{1\le i,j,\le k}
$$
is the Vandermonde determinant, $a_\de(z^{-1})$ is the similar
determinant for the variables $z^{-1}=(z_1^{-1},{\ldots}
,z_k^{-1})$,  $l(\la)$ is the length of the diagram $\la$ (the number of
nonzero rows), $s_\la(z^{-1})$ is the Schur function
calculated at the variables $z^{-1}$, and
$s_\la$ is the Schur function as an element of $\La$ identified
with $\H[0]$. Thus we have
$$
e_{-i_1}{\ldots} e_{-i_k}\Om_{-2k}=(-1)^{k(k-1)/2}\cdot[1]\left(\prod_{j=1}^kz_j^{k-i_j}
a_\de(z)a_\de(z^{-1})
\sum_\la s_\la(z^{-1})s_\la\right).
$$

For convenience, set $\te_p:=e_{-(k-p)}$, $0\le p\le k$. Given
$0\le\al_1,{\ldots},
\al_k\le k$, we have
\beq\label{comp1}
\te_{\al_1}{\ldots}\te_{\al_k}\Om_{-2k}
=[1]\left(\prod_{j=1}^kz_j^{\al_j}a_\de(z)\sum_{l(\la)\le
k}a_{\la+\de}(z^{-1})s_\la\right),
\eeq
where $a_{\la+\de}(x)=\det[x_i^{\la_j+k-j}]_{1\le i,j\le k}=s_\la(x)a_\de(x)$. Consider
a Young diagram $\mu=(\mu_1,{\ldots}
,\mu_k)= (0^{r_0}1^{r_1}2^{r_2}{\ldots} )$. Let us
sum~\eqref{comp1}
over all different permutations $\al=(\al_1,{\ldots}
,\al_k)$ of the sequence
$(\mu_1,{\ldots} ,\mu_k)$. Note that
the operators $e_j$ commute with each other, so that the
left-hand side does not depend on the order of the factors. In the
right-hand side, $\sum_\al\prod z_j^{\al_j}=m_\mu(z)$, a monomial
symmetric function. Thus we have
\beq\label{comp2}
\frac{k!}{\prod_{j=0}^k r_j!}\te_{\mu_1}{\ldots} \te_{\mu_k}=
[1]\left(m_\mu(z)a_\de(z)\sum_{l(\la)\le
k}a_{\la+\de}(z^{-1})s_\la\right).
\eeq
Let $\nu$ be a Young diagram with
at most $k$ rows and at most $k$ columns, i.e., $\nu\subset(k^k)$. We have
\beq\label{Kostka}
s_\nu(z)=\sum_{\mu}
K_{\nu\mu}m_{\mu}(z),
\eeq
where $K_{\nu\mu}$ are Kostka numbers. It is well known that
$K_{\nu\mu}=0$ unless $\mu\le\nu$, where
$\le$ is the standard ordering on partitions:
$\mu\le\nu\iff \mu_1+{\ldots} +\mu_i\le\nu_1+{\ldots} +\nu_i$ for
every $i\ge1$. In particular, $\mu_1\le \nu_1\le k$.
Besides, since we consider
only $k$ nonzero variables $z_1,{\ldots} ,z_k$, it also follows that
$m_\mu(z)=0$ unless $l(\mu)\le k$. Thus the sum in~\eqref{Kostka} can
be taken only over diagrams $\mu\subset(k^k)$, for which equation~\eqref{comp2}
holds. Multiplying this equation by $K_{\nu\mu}$ and summing
over $\mu$ yields
$$
\sum_{\mu=(0^{r_0}1^{r_1}2^{r_2}{\ldots} )\subset(k^k)}
\frac{k!}{\prod_{j=0}^k r_j!}K_{\nu\mu}\te_{\mu_1}{\ldots} \te_{\mu_k}=
[1]\left(s_\nu(z)a_\de(z)\sum_{l(\la)\le
k}a_{\la+\de}(z^{-1})s_\la\right).
$$
By the orthogonality relations, the right-hand side is equal to
$k!s_\nu$,  and the desired formula~\eqref{gensegal} follows.
\end{proof}

\subsection{Examples}
\cellsize=1em

In this section, we present the results of computing $\Phi(\tau)$ for the serpentine tableaux with $r(\tau)\le4$ (note that although the conditions of Theorem~\ref{th2} do not determine the isomorphism uniquely, these relations hold for {\it any} isomorphism satisfying them) in terms of Newton's power sums $p_{k}$. We write down only the ``nontrivial'' part of a tableau, meaning that it should be continued up to an infinite tableau in the ``serpentine'' way.
We also omit the normalizing coefficients of $\Phi(\tau)$, which are their norms in the inner product~\eqref{inner}.

\begin{center}
\bigskip
\begin{tabular}{|c|c|c||c|c|c|}
\hline
$r(\tau)$ & $\tau$ & $\Phi(\tau)$&$r(\tau)$ & $\tau$ & $\Phi(\tau)$\\
&  & {\footnotesize up to a constant} & &  & {\footnotesize up to a constant}\\
\hline

0&$\tau_{0}$ & $1=s_{\emptyset}$ &
4 &$\tau_{2}=\lower0
\cellsize\vbox{\footnotesize
\cput(1,1){1}
\cput(1,2){2}
\cput(1,3){3}
\cput(1,4){4}
\cells{
 _ _ _ _
|_|_|_|_|}}$& $p_{1}^{4}+3p_{2}^{2}-4p_{1}p_{3}=s_{(2^{2})}$\\
\hline

1&$\tau_{1}=\lower0
\cellsize\vbox{\footnotesize
\cput(1,1){1}
\cput(1,2){2}
\cells{
 _ _
|_|_|}}$& $p_{1}=s_{(1)}$&
4&$\lower0
\cellsize\vbox{\footnotesize
\cput(1,1){1}
\cput(1,2){3}
\cput(1,3){4}
\cput(2,1){2}
\cput(2,2){5}
\cput(2,3){6}
\cells{
 _ _ _
|_|_|_|
|_|_|_|}}$& $p_{1}^{4}-3p_{2}^{2}+2p_{1}p_{3}$

\\
\hline

2&$\lower0
\cellsize\vbox{\footnotesize
\cput(1,1){1}
\cput(1,2){2}
\cput(1,3){4}
\cput(2,1){3}
\cells{
 _ _ _
|_|_|_|
|_|}}$& $p_{2}$&
4&$\lower0
\cellsize\vbox{\footnotesize
\cput(1,1){1}
\cput(1,2){2}
\cput(1,3){4}
\cput(1,4){6}
\cput(2,1){3}
\cput(2,2){5}
\cput(2,3){7}
\cput(2,4){8}
\cells{
 _ _ _ _
|_|_|_|_|
|_|_|_|_|}}$& $p_{1}^{4}+12p_{2}^{2}+32p_{1}p_{3}$
\\
\hline

2&$\lower0
\cellsize\vbox{\footnotesize
\cput(1,1){1}
\cput(1,2){2}
\cput(2,1){3}
\cput(2,2){4}
\cells{
 _ _
|_|_|
|_|_|}}$& $p_{1}^{2}$&
4&$\lower0
\cellsize\vbox{\footnotesize
\cput(1,1){1}
\cput(1,2){3}
\cput(1,3){4}
\cput(1,4){6}
\cput(2,1){2}
\cput(2,2){5}
\cells{
 _ _ _ _
|_|_|_|_|
|_|_|}}$& $p_{1}^{2}p_{2}-p_{4}$
\\
\hline

3&$\lower0
\cellsize\vbox{\footnotesize
\cput(1,1){1}
\cput(1,2){2}
\cput(1,3){3}
\cput(2,1){4}
\cells{
 _ _ _
|_|_|_|
|_|}}$& $p_{1}^{3}-p_{3}$&
4&$\lower0
\cellsize\vbox{\footnotesize
\cput(1,1){1}
\cput(1,2){2}
\cput(1,3){4}
\cput(1,4){6}
\cput(1,5){8}
\cput(2,1){3}
\cput(2,2){5}
\cput(2,3){7}
\cells{
 _ _ _ _ _
|_|_|_|_|_|
|_|_|_|}}$& $p_{1}^{2}p_{2}+4p_{4}$
\\
\hline

3&$\lower0
\cellsize\vbox{\footnotesize
\cput(1,1){1}
\cput(1,2){2}
\cput(1,3){4}
\cput(1,4){6}
\cput(2,1){3}
\cput(2,2){5}
\cells{
 _ _ _ _
|_|_|_|_|
|_|_|}}$& $p_{1}^{3}+8p_{3}$&&&\\
\hline

3&$\lower0
\cellsize\vbox{\footnotesize
\cput(1,1){1}
\cput(1,2){2}
\cput(1,3){4}
\cput(2,1){3}
\cput(2,2){5}
\cput(2,3){6}
\cells{
 _ _ _
|_|_|_|
|_|_|_|}}$& $p_{1}p_{2}$&&&\\
\hline

\end{tabular}
\end{center}

\end{document}